\documentclass[11pt,a4paper]{scrartcl}

\usepackage[utf8]{inputenc}
\usepackage[T1]{fontenc}
\usepackage{amsmath}
\usepackage{amssymb}
\usepackage{amsthm}
\usepackage{mathtools}
\usepackage{thmtools}
\usepackage{enumitem}
\setlist[enumerate]{label=(\alph*),left=0pt,labelsep=1em}
\usepackage{multicol}
\usepackage{array}
\usepackage{xcolor}

\setlist[description]{labelindent=\parindent}
\usepackage[left=2cm,right=2cm,top=2cm,bottom=2.5cm]{geometry}

\KOMAoptions{parskip=never}

\AtBeginDocument{
	\setlength{\abovedisplayskip}{5pt}
	\setlength{\belowdisplayskip}{5pt}
	\setlength{\abovedisplayshortskip}{2pt}
	\setlength{\belowdisplayshortskip}{2pt}
	\setlength{\jot}{2pt}
}

\usepackage{hyperref}
\hypersetup{
	colorlinks,
	linkcolor={magenta!70!black},
	citecolor={green!60!black},
	urlcolor={blue!60!black}
}
\usepackage[capitalise]{cleveref}

\addtokomafont{disposition}{\rmfamily}  %same font everywhere

\usepackage{mleftright} 
\mleftright % redefine \left as \mleft and \right as \mright.
\usepackage[
backend=biber,
style=alphabetic,
sorting=nyt,
isbn=false
]{biblatex}
\renewbibmacro*{doi+eprint+url}{%if DOI is available, no URL
	\iftoggle{bbx:doi}
	{\printfield{doi}}
	{}%
	\newunit\newblock
	\ifboolexpr{togl {bbx:eprint} and test {\iffieldundef{doi}}}
	{\usebibmacro{eprint}}
	{}%
	\newunit\newblock
	\ifboolexpr{togl {bbx:url} and test {\iffieldundef{doi}}  and test {\iffieldundef{eprint}}}
	{\usebibmacro{url+urldate}}
	{}}
\usepackage{tikz}
\usetikzlibrary{calc}
\usepackage{tikz-cd}

\usepackage{todonotes}

\newcommand{\setc}{\mathsf{Set}}
\newcommand{\finset}{\mathsf{FinSet}}

\newcommand{\meas}{\mathsf{Meas}}
\newcommand{\borelmeas}{\mathsf{BorelMeas}}
\newcommand{\borelstoch}{\mathsf{BorelMeas}}

\newcommand{\N}{\mathbb{N}}

\newcommand{\Q}{\mathbb{Q}}
\newcommand{\R}{\mathbb{R}}

\newcommand{\id}{\operatorname{id}}

\newcommand{\dx}{\mathrm{d}}

\newcommand{\finiteset}[1]{[#1]}	% finite set {1, .., n} or perhaps {0, .., n-1}, depending on which convention is more convenient
\newcommand{\two}{\finiteset{2}}		% two-element set
\newcommand{\three}{\finiteset{3}}		% three-element set

\newcommand{\dist}{\mathcal{D}}
\newcommand{\distfin}{\mathcal{D}_{\mathrm{fin}}}
\newcommand{\giry}{\mathcal{P}}
\newcommand{\Sym}{\operatorname{Sym}}

\newtheorem{theo}{Theorem}[section]
\newtheorem{defi}[theo]{Definition}
\newtheorem{prop}[theo]{Proposition}
\newtheorem{lem}[theo]{Lemma}
\newtheorem{cor}[theo]{Corollary}

\theoremstyle{definition}
\newtheorem{rem}[theo]{Remark}
\newtheorem*{rem*}{Remark}

\Crefname{theo}{Theorem}{Theorems}
\Crefname{defi}{Definition}{Definitions}
\Crefname{prop}{Proposition}{Propositions}
\Crefname{lem}{Lemma}{Lemmas}
\Crefname{cor}{Corollary}{Corollaries}
\Crefname{conj}{Conjecture}{Conjectures}
\Crefname{rem}{Remark}{Remarks}
\Crefname{ex}{Example}{Examples}
\Crefname{equation}{}{}		% write '(2.3)' instead of 'Equation (2.3)'
\Crefname{figure}{Figure}{Figures}

\title{Naturality characterizes product measures and relative product measures}
\author{Victor Bloch, Tobias Fritz\\[0.5em]
\small Department of Mathematics, University of Innsbruck}

\date{\today}

\begin{document}

\maketitle

\begin{abstract}
	We show that the formation of the product measure is the only \emph{natural} way to assign to each tuple of probability measures on measurable spaces a probability measure on the product space.
	Here, naturality is meant in the sense of category theory and amounts to the condition that the assignment commutes with pushforward along measurable maps.
	In the standard Borel setting, we also prove an analogous result for relative products over a fixed base probability space.
	As a corollary, we conclude that the formation of relative products is a lax symmetric monoidal functor.
\end{abstract}

\tableofcontents

\section{Introduction}

Categorical approaches to probability theory and statistics trace back to Lawvere's seminal 1962 manuscript on the category of probabilistic mappings~\cite{Lawvere1962} as well as Giry's 1982 paper on the monad of probability measures~\cite{Giry1982}.
But especially in recent years, there has been a surge of interest with surprisingly strong results.
To name a few, these include fully categorical proofs of de Finetti's theorem~\cite{Fritz2018}, the Aldous--Hoover theorem~\cite{Chen2025}, the $d$-separation criterion~\cite{FritzKlingler2023}
and categorical accounts of ergodic decomposition~\cite{Moss2023}, the law of large numbers~\cite{Fritz2026} and the Metropolis--Hastings algorithm~\cite{Cornish2026}.
A recurring theme is that probabilistic structures are abstractly encoded in categorical terms, and this leads to a higher-level understanding of their properties and interrelations,
and sometimes also to intriguing new questions.

In the present work, we raise and answer a question of this kind: what singles out the product-measure construction---and its generalization to relative products over a fixed base probability space---among all possible ways of combining two probability measures into a joint distribution?
The simple answer is that the categorical notion of \emph{naturality} completely characterizes it, and does so for any fixed number of factors.
In measure-theoretic terms, naturality means that the measure associated to a tuple of pushforward measures is the pushforward of the measure associated to the original tuple.
We show in \cref{theo:onenaturaltransformation} that taking product measures is the only way to achieve this.
This is arguably not very surprising, and indeed also closely related to a result in copula theory that we discuss below.

Product measures are a fundamental concept in probability theory associated with stochastic independence.
The more involved concept of \emph{conditional independence} is likewise related to the construction of \emph{relative product measures} over a fixed base probability space,
where the measure on this base space may be thought of as formalizing shared information between the two spaces above it.
As is the case for categorical approaches to topology and geometry as well, such relative constructions are particularly amenable to a categorical treatment via slice categories.
In fact, we will extend \cref{theo:onenaturaltransformation} to a relative version in \cref{theo:onenaturaltransformation_relative}.
It characterizes the relative product measure as the uniquely natural way of combining a tuple of probability measures over a fixed base probability space into a joint distribution.
Since the definition of relative product measures involves regular conditional probabilities, this result is limited to standard Borel spaces.

As a first application of this result,
we conclude in \cref{cor:relativeproductlaxmonoidal} that the relative product measure construction is a lax symmetric monoidal functor.

\subsection*{Related work}

The present work is closely related to \emph{copula theory},
which studies systematic ways of combining marginal distributions into joint distributions~\cite{Nelsen2006,Durante2007,Durante2009}.
If one restricts to atomless distributions on standard Borel spaces,
then the probability spaces one works with are all isomorphic to the unit interval $[0,1]$ with the uniform Lebesgue measure,
Hence an $n$-variate copula is defined as a probability measure on $[0,1]^n$ with uniform marginals.
For example, the \emph{independence copula} is the product measure on~$[0,1]^n$.

Close to our results is the characterization, due to Durante, Sarkoci and Sempi, of the bivariate independence copula as the unique bivariate copula that is invariant under all measure-preserving transformations of its coordinates~\cite[Theorem~10]{Durante2009}.
In fact, this can be viewed as a version of our \cref{theo:onenaturaltransformation},
namely the restriction to uncountable standard Borel spaces and considering the functor which maps every such space to the non-atomic probability measures on it.
The proof amounts to noting that the Lebesgue measure on $[0,1]^2$ is the only measure that assigns the same value to every piece of the usual $k \times k$ partition into smaller squares.
The same characterization and proof apply to the $n$-variate independence copula too.

There also is other work on natural transformations between functors involving probability measures.
Similar in spirit to our characterization of product measures is~\cite[Corollary~A.3.10]{Bartels2004},
where Bartels has characterized maps $(\dist X)^n \to \dist(X^m)$ which are natural in $X \in \setc$ with $n,m \in \N$ arbitrary but fixed.
Ongoing work by Han, Wang and the second author aims to characterize natural transformations $\giry \Rightarrow \giry\giry$, where $\giry : \borelmeas \to \borelmeas$ is the usual functor of probability measures on standard Borel spaces.

Finally, some of the arguments in this manuscript can be formulated and proved elegantly in terms of the abstract framework of Markov categories~\cite{Fritz2020}.
We nevertheless prefer to spell out the measure-theoretic details here in order to avoid the additional overhead of setting up the Markov-categorical framework,
given that our main results are about the classical functors of probability measures only.

\subsection*{Acknowledgements}

The first author has been supported by a doctoral fellowship granted by the University of Innsbruck.

\section{Preliminaries}

\subsection{Basics}

In this subsection, we declare some notation and recall some standard material on functors of probability measures.
For $n \in \N$, we write $\finiteset{n} \coloneqq \{1,\ldots,n\}$.
For natural transformations, we usually omit the subscripts which index the components, as these can be inferred from context.

Denote by $\setc$ the category of sets and functions. Let $\dist: \setc \rightarrow \setc$ be the \emph{distribution functor}. It sends a set $X$ to the set $\dist X$ of discrete probability measures on $X$.
What we mean by a discrete probability measure is a finitely supported normalized measure on the power set $\sigma$-algebra $2^X$,
or equivalently a finitely supported function $p: X \rightarrow [0,1]$ with $\sum_{x \in X} p(x) = 1$.
On functions $f : X \to Y$, $\dist$ acts by taking the pushforward, $(Df)(p) \coloneq f_\ast p$, where $(f_\ast p)(S) = p(f^{-1}(S))$ for all $S \subseteq Y$.
Functoriality of $\dist$ is straightforward to check.
We also write $\distfin : \finset \rightarrow \setc$ for the restriction of~$\dist$ to $\finset$, the full subcategory of $\setc$ consisting of finite sets.

A broader category for measure-theoretic probability is given by $\meas$, the category of measurable spaces and measurable maps. A probability measure on a measurable space $(X,\Sigma_X)$ is a function $p: \Sigma_X \rightarrow [0,1]$ that satisfies $\sigma$-additivity and the normalization equation $p(X) = 1$.
For a measurable map $f:(X,\Sigma_X) \rightarrow (Y,\Sigma_Y)$, the pushforward $f_\ast p$ is again defined by $(f_\ast p)(S) \coloneq p(f^{-1}(S))$ for $S \in \Sigma_Y$. The analog of $\dist$ now is the \emph{Giry functor} $\giry: \meas \rightarrow \setc$~\cite{Giry1982}.
This functor sends a measurable space $(X,\Sigma_X)$ to the set of all probability measures on it.\footnote{As for the Giry monad~\cite{Giry1982}, we could also endow every $\giry X$ with the $\sigma$-algebra generated by the evaluation maps and therefore make $\giry$ land in $\meas$ itself,
but we will not need this here.}
On measurable maps, $\giry$ acts as the pushforward, that is, for $f: (X,\Sigma_X) \rightarrow (Y,\Sigma_Y)$ and $p$ a probability measure on $(X,\Sigma_X)$ we have $(\giry f) (p) \coloneq f_\ast p$.

A measurable space $(X,\Sigma_X)$ is a \emph{standard Borel space} if there is a separable and completely metrizable topology on $X$ such that $\Sigma_X$ is the $\sigma$-algebra generated by its open sets.
Standard Borel spaces play a fundamental role in probability theory, as they are both very well-behaved and sufficiently general to include most measurable spaces of interest in applications.
We denote the full subcategory of $\meas$ on standard Borel spaces by $\borelmeas$.
The Giry functor restricts to a functor on $\borelmeas$, and we also write $\giry: \borelmeas \rightarrow \setc$ for this restriction.

\subsection{Product measures}

We fix an arity $n \in \N$ throughout.
For measurable spaces $(X_1,\Sigma_{X_1}), \dots, (X_n,\Sigma_{X_n})$,
we equip $\prod_{i=1}^n X_i$  with the \emph{product $\sigma$-algebra} $\bigotimes_{i=1}^n \Sigma_{X_i}$,
which is generated by the rectangles $\prod_{i=1}^n S_i$ with $S_i \in \Sigma_{X_i}$.
This is the categorical product in $\meas$, and also in $\borelmeas$ when all the $X_i$ are standard Borel, with respect to the usual projection maps
\[
  \pi_i \: : \: \prod_{j=1}^n X_j \longrightarrow X_i.
\]
This product clearly restricts to the usual categorical product on $\setc$ and $\finset$.
For $n=0$, the empty product space is a singleton with its unique $\sigma$-algebra.
For probability measures $p_i \in \giry X_i$, a \emph{joint distribution} is a probability measure $r$ on $\prod_{i=1}^n X_i$ such that the $p_i$ are its marginals, that is,
\[
	\pi_{i,\ast} r = p_i \qquad \forall i \in \finiteset{n}.
\]
For $n \geq 2$, there are typically many such joint distributions $r$, which differ in the stochastic dependence among their coordinates.
The canonical choice, corresponding to mutual independence, is the product measure:

\begin{prop}[Product measure] \label{defi:productmeasure}
  Let $n \in \N$ and let $p_i \in \giry X_i$ be probability measures on measurable spaces $(X_i,\Sigma_{X_i})$ for $i \in \finiteset{n}$.
  Then there is a unique probability measure $\bigotimes_{i=1}^n p_i$ on $\prod_{i=1}^n X_i$ such that
  \begin{equation}
    \label{eq:productmeasure}
    (p_1 \otimes \cdots \otimes p_n)(S_1 \times \cdots \times S_n)
    = p_1(S_1) \cdots p_n(S_n)
  \end{equation}
  for all $S_i \in \Sigma_{X_i}$.
\end{prop}

This is a basic result in measure theory usually proven via Carathéodory's extension theorem.
For every $n \in \N$, the formation of product measures therefore defines maps
\begin{align*}
  \nabla \: : \: \giry X_1 \times \cdots \times \giry X_n & \longrightarrow \giry(X_1 \times \cdots \times X_n) \\
  (p_1,\dots,p_n) & \longmapsto p_1 \otimes \cdots \otimes p_n
\end{align*}
which are natural in $X_1,\dots,X_n$, meaning that for all measurable $f_i : X_i \to X_i'$, we have
\[
	(f_1 \times \cdots \times f_n)_\ast (p_1 \otimes \cdots \otimes p_n)
	= f_{1,\ast} p_1 \otimes \cdots \otimes f_{n,\ast} p_n.
\]
To see this, we evaluate on any measurable rectangle $\prod_{i=1}^n S_i'$ with $S_i' \in \Sigma_{X_i'}$, where we get
\begin{align*}
	(f_{1,\ast} p_1 \otimes \cdots \otimes f_{n,\ast} p_n)(S_1' \times \cdots \times S_n')
  &= p_1(f_1^{-1}(S_1')) \cdots p_n(f_n^{-1}(S_n')) \\[2pt]
  &= (p_1 \otimes \cdots \otimes p_n)\bigl(f_1^{-1}(S_1') \times \cdots \times f_n^{-1}(S_n')\bigr) \\[2pt]
	&= \bigl((f_1 \times \cdots \times f_n)_\ast (p_1 \otimes \cdots \otimes p_n)\bigr)(S_1' \times \cdots \times S_n').
\end{align*}
This implies the desired equality by uniqueness of the product measure in the sense of \cref{defi:productmeasure}.
Hence the formation of product measures can be seen as a natural transformation
\begin{equation}
	\label{eq:productnaturality_giry}
  \nabla \: : \: \giry\_ \times \cdots \times \giry\_ \Longrightarrow \giry(\_ \times \cdots \times \_),
\end{equation}
where both sides are functors $\meas^{\times n} \longrightarrow \setc$.\footnote{In fact, these natural transformations make $\giry$ into a lax symmetric monoidal functor~\cite{Fritz2018}.}
All of this applies likewise to the distribution functor~$\dist$ on $\setc$: the formation of product distributions is a natural transformation
\begin{equation}
	\label{eq:productnaturality_dist}
  \nabla \: : \: \dist\_ \times \cdots \times \dist\_ \Longrightarrow \dist(\_ \times \cdots \times \_)
\end{equation}
between functors $\setc^n \longrightarrow \setc$.
Restricting to $\finset$ shows that it also holds for $\distfin$.

\subsection{Relative product measures}

Product measures are about probabilistic independence: the coordinate projections on the probability space $(\prod_{i=1}^n X_i,\bigotimes_{i=1}^n p_i)$ are mutually independent,
and the product measure is the unique joint distribution of the $p_i$ with this property.
But even more important is the concept of \emph{conditional independence}.
The relevant notion of product measure is the \emph{relative product measure}~\cite{Dawid1999},
which we now review.
These relative products have been studied under various names in different areas,
e.g.~in the context of graphical models~\cite{Dawid1999} and ergodic theory~\cite[Definition~6.15]{Einsiedler2011}.
Categorical treatments include~\cite{Flori2016,Simpson2018,Perrone2025} and~\cite[Definition~12.8]{Fritz2020}.\footnote{These works partly restrict to the binary case, and partly also to the case where the structure maps to the base are product projections. Modulo adding events of measure zero, the latter is no restriction.}

Similarly to how product measures live on product spaces, relative product measures live on fiber products.
For measurable maps
\[
	\rho_{X_1} \: : \: X_1 \longrightarrow B,
	\qquad
	\dots,
	\qquad
	\rho_{X_n} \: : \: X_n \longrightarrow B,
\]
we endow the fiber product
\begin{equation}
  \label{eq:pullbackdiagram2}
  X_1 \times_B \cdots \times_B X_n
  \coloneqq \bigl\{(x_1,\dots,x_n) \in X_1 \times \cdots \times X_n : \rho_{X_1}(x_1)=\cdots=\rho_{X_n}(x_n)\bigr\}
\end{equation}
with the subspace $\sigma$-algebra induced from the product measurable space $\prod_{i=1}^n X_i$. 
Together with the coordinate projections $\pi_i$,
this object is a (wide) pullback in $\meas$, as illustrated by the diagram
\begin{equation}
	\label{eq:pullbackdiagram}
  \begin{tikzcd}[column sep=large]
		& X_1 \times_B \cdots \times_B X_n \arrow[dl, "\pi_1"'] \arrow[dr, "\pi_n"] \\
		X_1 \ar[dr, "\rho_1"'] & \cdots & X_n \ar[dl, "\rho_n"] \\
    & B
  \end{tikzcd}
\end{equation}
Then the question is: If given probability measures $p_i \in \giry X_i$ satisfy
\[
	\rho_{X_1,\ast} p_1 = \dots = \rho_{X_n,\ast} p_n \eqqcolon \omega,
\]
is there a canonical probability measure $r$ on the pullback such that $\pi_{i,\ast} r = p_i$ for all $i$?
If $B$ and the~$X_i$ are standard Borel, then so is the pullback,
and in this case there is a unique such $r$ for which the maps to the $X_1,\dots,X_n$ are mutually conditionally independent given $B$.\footnote{For an elegant proof using Markov categories, see~\cite[Proposition~12.9]{Fritz2020}. While this treats the binary case only, it straightforwardly	generalizes to arbitrary $n$.}
This is the relative product of $p_1,\dots,p_n$ over $\omega$, and we denote it by
\begin{equation}
	\label{eq:relativeproductmeasure_type}
  p_1 \otimes_\omega \cdots \otimes_\omega p_n \:\in\: \giry(X_1 \times_B \cdots \times_B X_n).
\end{equation}
Writing this measure down explicitly requires regular conditional probabilities.
These are most easily understood as Markov kernels, which we now recall.

\begin{defi}
	For measurable spaces $(X, \Sigma_X)$ and $(Y, \Sigma_Y)$, a \emph{Markov kernel} from $X$ to $Y$ is a map
	\begin{align*}
		\kappa \: : \: \Sigma_Y \times X & \longrightarrow [0,1] \\
		(S, x) & \longmapsto \kappa(S|x)
	\end{align*}
	such that:
	\begin{enumerate}[label=(\roman*)]
		\item $\kappa(-|x)$ is a probability measure on $(Y, \Sigma_Y)$ for all $x \in X$.
		\item $\kappa(S|\cdot)$ is measurable for all $S \in \Sigma_Y$.
	\end{enumerate}
\end{defi}

Equivalently, a Markov kernel can be viewed as a measurable map $X \to \giry Y$, where $\giry Y$ is equipped with the usual $\sigma$-algebra generated by the evaluation maps.\footnote{The equivalence of these two points of view leads to the definition of \emph{representable Markov category}~\cite[Definition~3.10]{Fritz2023}.}

\begin{defi}
	Let $\rho_X : X \to B$ be a measurable map between arbitrary measurable spaces.
	Then a \emph{regular conditional probability} of a probability measure $p \in \giry X$ with respect to $\rho_X$ is a Markov kernel $p(-|\cdot) : \Sigma_X \times B \to [0,1]$ such that
  \begin{equation}
		\label{eq:disintegration}
		p(S \cap \rho_X^{-1}(T)) = \int_{b \in T} p(S|b) \  \omega(\dx b)
		\qquad
		\forall S \in \Sigma_X, T \in \Sigma_B,
	\end{equation}
	where $\omega \coloneq \rho_{X,\ast} p$.
\end{defi}

\begin{prop}\label{prop:regularconditionalprobability}
	If $X$ and $B$ are standard Borel, then such a regular conditional probability exists for every $p \in \giry X$,
	and it is unique up to $\omega$-a.e.~equality.
\end{prop}

See e.g.~\cite[Lemma 10.4.3 \& Proposition 10.4.12]{Bogachev2007} for a proof of this standard result.
To understand the definition, it is also worth keeping in mind that
\begin{equation}
	\label{eq:fiberconcentration}
	p(\rho_X^{-1}(b)|b) = 1
\end{equation}
for $\omega$-almost all $b \in B$~\cite[Example~10.4.11]{Bogachev2007}.
In other words, $p(-|\cdot)$ is essentially a stochastic section of $\rho_X$.

\begin{defi}
	\label{defi:relativeproductmeasure}
  Let $X_1,\dots,X_n$ and $B$ be standard Borel spaces with measurable maps $\rho_{X_i} : X_i \to B$, and fix $\omega \in \giry B$.
	Suppose that measures $p_i \in \giry X_i$ satisfy $\rho_{X_i,\ast} p_i = \omega$ for every $i$.
  Then the \emph{relative product} of $p_1,\dots,p_n$ over $\omega$ is the probability measure on $\prod_{i=1}^n X_i$ defined by
  \begin{equation}
		\label{eq:relativeproductmeasure}
    (p_1 \otimes_\omega \cdots \otimes_\omega p_n)(S)
    \coloneqq \int_{b \in B} \bigl(p_1(-|b) \otimes \cdots \otimes p_n(-|b)\bigr)(S)\,\omega(\dx b)
    \qquad \forall S \in \Sigma_{X_1 \times \cdots \times X_n}.
  \end{equation}
\end{defi}

The integrand is measurable, first for rectangles and then for all measurable sets by the $\pi$-$\lambda$ theorem.
The uniqueness of regular conditional probabilities up to $\omega$-a.e.~equality makes the integral independent of their choice, and the monotone convergence theorem implies $\sigma$-additivity.
Thus the definition gives a well-defined probability measure.
For $n=1$, it recovers $p_1$ by the definition~\eqref{eq:disintegration}.

If $B$ is finite, then the relative product measure is characterized in analogy with~\eqref{eq:productmeasure} by the formula
\[
  (p_1 \otimes_\omega \cdots \otimes_\omega p_n)(S_1 \times \cdots \times S_n)
  = \sum_{b \in B \: : \: \omega(b)>0}
    \frac{p_1(S_1 \cap \rho_{X_1}^{-1}(b)) \cdots p_n(S_n \cap \rho_{X_n}^{-1}(b))}{\omega(b)^{n-1}}
  \qquad \forall S_i \in \Sigma_{X_i}.
\]
In particular, if the $X_i$ are also discrete, then the probabilities of individual outcomes are given by
\[
  (p_1 \otimes_\omega \cdots \otimes_\omega p_n)(x_1,\dots,x_n)
  = \begin{cases}
		\displaystyle\frac{p_1(x_1) \cdots p_n(x_n)}{\omega(b)^{n-1}}
			& \text{if } \omega(b) > 0 \text{ for } b = \rho_{X_1}(x_1) = \dots = \rho_{X_n}(x_n), \\[6pt]
					0 & \text{if the } \rho_{X_i}(x_i) \text{ are not all equal or if } \omega(b) = 0.
  \end{cases}
\]
The second case is a manifestation of the following general support property,
which explains how \cref{defi:relativeproductmeasure} matches up with~\eqref{eq:relativeproductmeasure_type}.

\begin{lem}
	\label{lem:relativeproductsupport}
  The relative product measure $p_1 \otimes_\omega \cdots \otimes_\omega p_n$ is supported on the fiber product $X_1 \times_B \cdots \times_B X_n$,
  considered as a measurable subset of $\prod_{i=1}^n X_i$.
\end{lem}

For the binary case, see e.g.~\cite[Prop.~6.16(1)]{Einsiedler2011}.

\begin{proof}
  Since $B$ is standard Borel, the diagonal $\Delta_B \coloneqq \{(b,\dots,b):b\in B\}$ in $B^n$ is measurable. Hence
  \[
    X_1 \times_B \cdots \times_B X_n
    = (\rho_{X_1} \times \cdots \times \rho_{X_n})^{-1}(\Delta_B)
  \]
  is a measurable subset of $\prod_{i=1}^n X_i$.
  Since there are only finitely many factors, the fiber concentration property~\eqref{eq:fiberconcentration} gives
  \[
    p_i(\rho_{X_i}^{-1}(b)|b)=1 \qquad \forall i \in \finiteset{n}
  \]
  away from an $\omega$-null set.
  For each such $b$, the product measure $\bigotimes_{i=1}^n p_i(-|b)$ is concentrated on
  \[
    \rho_{X_1}^{-1}(b) \times \cdots \times \rho_{X_n}^{-1}(b) \:\subseteq\: X_1 \times_B \cdots \times_B X_n.
  \]
  Thus, with $S=X_1 \times_B \cdots \times_B X_n$ in~\eqref{eq:relativeproductmeasure}, the integrand is a.e.~equal to $1$,
  and hence the integral is $1$ as well.
\end{proof}

We now turn to further categorical aspects of relative product measures.
Given that the construction takes place ``over $B$'', it is natural to work with slice categories like $\borelmeas/B$.
An object of $\borelmeas/B$ is a pair $(X, \rho_X)$ where $X$ is a standard Borel space and $\rho_X: X \rightarrow B$ is a measurable map.
We often suppress the map $\rho_X$ in the notation and simply write $X$, thinking of $\rho_X$ as extra structure on the measurable space $X$.
In this slice category, the fiber product~\eqref{eq:pullbackdiagram2}, with the projections in~\eqref{eq:pullbackdiagram}, is the categorical product of $X_1,\dots,X_n$.
In particular, the empty product is the terminal object $(B,\id_B)$.

We also fix the base probability measure $\omega \in \giry B$ throughout.
Then the relative version of the probability measures functor is
\begin{align*}
	\giry^\omega \: : \: \borelmeas/B & \longrightarrow \setc \\
	(X, \rho_X) & \longmapsto \{ p \in \giry X \: : \: \rho_{X,\ast} p = \omega \}.
\end{align*}
The action on a morphism of $\borelmeas/B$, which is a commutative triangle of the form
\[
	\begin{tikzcd}
		X \arrow[rr, "f"] \arrow[dr, "\rho_X"'] & & Y \arrow[dl, "\rho_Y"] \\
		& B &
	\end{tikzcd}
\]
is again given by the pushforward $f_\ast$.

Our next goal is to show that the formation of relative product measures defines a natural transformation,
but this requires a bit of preparation.

\begin{lem}\label{lem:pushforwarddisintegration}
  Let $f : X \to Y$ be a morphism in $\borelmeas/B$.
	Then for every $p \in \giry^\omega X$, the Markov kernel
	\begin{align*}
		B & \longrightarrow \giry Y \\
		b & \longmapsto f_\ast p(-|b)
	\end{align*}
	is a regular conditional probability of $f_\ast p$ with respect to $\rho_Y$.
\end{lem}

\begin{proof}
	The displayed map is a Markov kernel because it is the composite of the Markov kernel
	$p(-|\cdot) : B \to \giry X$ with the measurable pushforward map $f_\ast : \giry X \to \giry Y$.
	Since $f$ is a morphism over $B$, we have
	\[
		\rho_{Y,\ast}(f_\ast p)
		= (\rho_Y \circ f)_\ast p
		= \rho_{X,\ast} p
		= \omega.
	\]
	Consequently, for $S \in \Sigma_Y$ and $T \in \Sigma_B$,
	\begin{align*}
		(f_\ast p)(S \cap \rho_Y^{-1}(T))
		&= p\bigl(f^{-1}(S) \cap \rho_X^{-1}(T)\bigr) \\
		&= \int_{b \in T} p(f^{-1}(S)|b) \, \omega(\dx b) \\
		&= \int_{b \in T} \bigl(f_\ast p(-|b)\bigr)(S) \, \omega(\dx b),
	\end{align*}
	which is precisely the disintegration identity~\eqref{eq:disintegration}.
\end{proof}

Here is now the promised naturality statement for relative product measures.

\begin{lem}
	\label{lem:relativeproductnaturality}
	For every $n \in \N$, the formation of relative product measures defines maps
	\begin{equation}
		\label{eq:relativeproductnaturality}
    \nabla^\omega \: : \: \giry^\omega X_1 \times \cdots \times \giry^\omega X_n \longrightarrow \giry^\omega(X_1 \times_B \cdots \times_B X_n)
	\end{equation}
	that are natural in $X_1,\dots,X_n \in \borelmeas/B$.
\end{lem}

For the distribution functor $\distfin$ in place of $\giry$, the binary case is also part of~\cite[Proposition~5.2.1]{Flori2016}.

\begin{proof}
  Let $p_i \in \giry^\omega X_i$ for every $i$.
  By \cref{lem:relativeproductsupport}, we can regard
  $r \coloneqq p_1 \otimes_\omega \cdots \otimes_\omega p_n$ as a probability measure on $X_1 \times_B \cdots \times_B X_n$.
  For every coordinate projection $\pi_i$ and $S \in \Sigma_{X_i}$,
  \[
		(\pi_{i,\ast} r)(S)
    = \int_{b \in B} p_i(S|b)\,\omega(\dx b)
    = p_i(S).
  \]
  Here we have used $p_j(X_j|b)=1$ for all $j \ne i$. Thus $r$ has the prescribed marginals, and
  \[
		(\rho_{X_i} \circ \pi_i)_\ast r = \rho_{X_i,\ast} p_i = \omega.
	\]
  Hence $r \in \giry^\omega(X_1 \times_B \cdots \times_B X_n)$, making $\nabla^\omega$ well-defined.

  For naturality, let $f_i : X_i \to Y_i$ be morphisms over $B$, so that the diagrams
  \[
    \begin{tikzcd}[column sep=large]
      X_i \arrow[rr, "f_i"] \arrow[dr, "\rho_{X_i}"'] & & Y_i \arrow[dl, "\rho_{Y_i}"] \\
      & B &
    \end{tikzcd}
  \]
	commute.
	Write $f \coloneqq \prod_{i=1}^n f_i$ and $r' \coloneqq f_{1,\ast} p_1 \otimes_\omega \cdots \otimes_\omega f_{n,\ast} p_n$.
  By \cref{lem:pushforwarddisintegration}, regular conditional probabilities of $f_{i,\ast} p_i$ are given by $b \mapsto f_{i,\ast} p_i(-|b)$.
  Therefore, for every measurable $S \subseteq \prod_{i=1}^n Y_i$, naturality of ordinary product measures gives
  \begin{align*}
    r'(S)
		&= \int_{b \in B} \bigl(f_{1,\ast} p_1(-|b) \otimes \cdots \otimes f_{n,\ast} p_n(-|b)\bigr)(S)\,\omega(\dx b) \\
    &= \int_{b \in B} \bigl(p_1(-|b) \otimes \cdots \otimes p_n(-|b)\bigr)(f^{-1}(S))\,\omega(\dx b) \\
		&= r(f^{-1}(S)) \\[2pt]
    &= (f_\ast r)(S).
  \end{align*}
  Since each $f_i$ is a morphism over $B$, the map $F$ restricts to
  \[
		f_1 \times_B \cdots \times_B f_n \: : \: X_1 \times_B \cdots \times_B X_n \longrightarrow X_1' \times_B \cdots \times_B X_n'.
	\]
  The support property of $r$ and $r'$ therefore gives
  \[
    (f_1 \times_B \cdots \times_B f_n)_\ast (p_1 \otimes_\omega \cdots \otimes_\omega p_n)
		= f_{1,\ast} p_1 \otimes_\omega \cdots \otimes_\omega f_{n,\ast} p_n,
  \]
  which is the desired naturality.
\end{proof}

So similar to the formation of product measures,
the formation of relative product measures can be seen as a natural transformation
\begin{equation}
	\label{eq:relativeproductnaturality_giry}
  \nabla^\omega \: : \: \giry^\omega \_ \times \cdots \times \giry^\omega \_ \Longrightarrow \giry^\omega(\_ \times_B \cdots \times_B \_),
\end{equation}
where both sides are functors $(\borelmeas/B)^n \longrightarrow \setc$.

\begin{rem}
	\begin{enumerate}
		\item 
			If $B$ is a singleton, then the relative product reduces to the ordinary product,
			and \cref{lem:relativeproductnaturality} reduces to the naturality of ordinary product measures.
		\item 
			Clearly all of this applies likewise to the distribution functor $\dist$ on $\setc$:
			for any finite set $B$ and $\omega \in \dist B$,
			the formation of relative product distributions is a natural transformation
			\begin{equation}
				\label{eq:relativeproductnaturality_dist}
        \nabla^\omega \: : \: \dist^\omega\_ \times \cdots \times \dist^\omega\_ \Longrightarrow \dist^\omega(\_ \times_B \cdots \times_B \_)
			\end{equation}
			between functors $(\setc/B)^n \longrightarrow \setc$, where $\dist^\omega X \coloneqq \{p \in \dist X : \rho_{X,\ast} p=\omega\}$.
      The same holds for $\distfin^\omega$ on $\finset/B$.
	\end{enumerate}
\end{rem}

Somewhat tangentially, we will now show that $\nabla^\omega$ is measurable,
provided that we equip the spaces $\giry^\omega X$ with the smallest $\sigma$-algebra which makes the evaluation maps
	\begin{align*}
		\label{eq:evalmap}
		\giry^\omega X & \longrightarrow [0,1] \\
		p & \longmapsto p(S)
	\end{align*}
	measurable for all $S \in \Sigma_X$.
	When $B$ is a singleton, this is the usual Giry $\sigma$-algebra on $\giry X$~\cite[Section~I.2(a)]{Giry1982},
	while in general $\giry^\omega X \subseteq \giry X$ is a Borel subset with the induced $\sigma$-algebra.
	Since $\giry X$ is standard Borel for standard Borel $X$, the same holds for $\giry^\omega X$.
	We can thus also view $\giry^\omega$ as a functor $\borelmeas/B \to \borelmeas$,
	where the measurability of the pushforward maps $f_\ast : \giry^\omega X \to \giry^\omega Y$ holds because it holds for $\giry : \meas \to \meas$.

\begin{lem}
	\label{rem:measurableproducts}
	With these $\sigma$-algebras, the relative product maps~\eqref{eq:relativeproductnaturality} are measurable.
\end{lem}

\begin{proof}
	We need to show that the map
	\begin{align*}
		\giry^\omega X_1 \times \cdots \times \giry^\omega X_n & \longrightarrow [0,1] \\
		(p_1,\dots,p_n) & \longmapsto (p_1 \otimes_\omega \cdots \otimes_\omega p_n)(S)
	\end{align*}
	is measurable for every measurable $S \subseteq X_1 \times \cdots \times X_n$.
	By a standard application of the $\pi$-$\lambda$ theorem,
	it is enough to check this for rectangles $S = S_1 \times \cdots \times S_n$.

	To see this, recall that if $q : A \to \giry^\omega X$ is measurable with $A$ standard Borel too,
	then there is a jointly measurable map $A \times B \to \giry X$ that is a version of the regular conditional probability for every $a \in A$.\footnote{This is arguably understood most cleanly in terms of conditionals in the Markov category $\borelstoch$, see~\cite[Proposition~11.7]{Fritz2020}. In measure-theoretic terms, the relevant result is~\cite[Theorem~1.25]{Kallenberg2017}.}
	Taking $A = \giry^\omega X$ itself and $q$ the identity map shows that the regular conditional probabilities $p(-|b)$ can be chosen jointly measurably in $b \in B$ and $p \in \giry^\omega X$.
	Hence the integrand in the construction of the relative product measure
	\[
		(p_1 \otimes_\omega \cdots \otimes_\omega p_n)(S_1 \times \cdots \times S_n) = \int_{b \in B} p_1(S_1|b) \cdots p_n(S_n|b)\,\omega(\dx b)
	\]
	depends measurably on $(p_1,\dots,p_n)$,
	being the product of the corresponding measurable functions $p_i \mapsto p_i(S_i|\cdot)$.
	But then by measurability of integration~\cite[Lemma~1.15(i)]{Kallenberg2017},
	also the integral depends measurably on $(p_1,\dots,p_n)$,
	as was to be shown.
\end{proof}

\section{Main results}

\subsection{Statements}

A natural question now is, are there any other natural transformations of the form~\eqref{eq:productnaturality_giry} or \eqref{eq:productnaturality_dist}? 
Or is the formation of product measures the only natural way to combine two probability measures into a joint measure?
This is the first question that we answer in the affirmative.

\begin{theo}
  \label{theo:onenaturaltransformation}
  \begin{enumerate}
  \item 
    With $\giry : \meas \to \setc$ the Giry functor and any $n \in \N$, the formation of product measures is the only family of maps
    \[
      \giry X_1
      \times \dots \times \giry X_n \longrightarrow \giry(X_1 \times \dots \times X_n)
    \]
    which is natural in $X_1,\dots,X_n \in \meas$.
  \item 
    The analogous statement for the distribution functor $\dist : \setc \to \setc$ (resp.~$\distfin : \finset \to \setc$) also holds.
  \end{enumerate}
\end{theo}

For relative product measures, we have to restrict to standard Borel spaces.
But otherwise the statement is the same.

\begin{theo}
  \label{theo:onenaturaltransformation_relative}
  \begin{enumerate}
  \item For all $n \in \N$, any standard Borel space $B$, and $\omega \in \giry B$, the formation of relative product measures with respect to $\omega$ is the only family of maps
    \[
      \giry^\omega X_1 \times  \cdots \times \giry^\omega X_n \longrightarrow \giry^\omega(X_1 \times_B \cdots \times_B X_n)
    \]
    which is natural in $X_1,\dots,X_n \in \borelmeas/B$.
  \item With $B$ a (finite) set, the analogous statement for $\dist^\omega : \setc/B \to \setc$ (resp.~$\distfin^\omega : \finset/B \to \setc$) also holds.
  \end{enumerate}
\end{theo}

Taking $B$ to be a singleton recovers the standard Borel version of \cref{theo:onenaturaltransformation}, which naturally holds as well.

\begin{rem}
  \begin{enumerate}
  \item 
    The proofs of both results will first consider the case of finite sets and then extend from there.
    This is designed such that the proofs go through for all of $\giry$, $\dist$ and $\distfin$ simultaneously.
    If one is interested in $\giry$ only,
    then simpler proofs are possible using the Lebesgue measure instead together with the fact that every other probability measure on a standard Borel space is a pushforward of it.
    This is essentially the approach taken in the existing copula theory result~\cite[Theorem~10]{Durante2009}.
  \item 
    In \cref{theo:onenaturaltransformation,theo:onenaturaltransformation_relative}, we have considered the functors $\giry$ and $\giry^\omega$ to take values in $\setc$.
    With the corresponding $\giry : \meas \to \meas$ and $\giry^\omega : \borelmeas/B \to \borelmeas$ versions of these functors,
    the same statements hold as corollaries:
    the product maps have measurable components by \cref{rem:measurableproducts},
		and uniqueness follows from the corresponding set-valued statements.
  \end{enumerate}
\end{rem}

As a straightforward consequence of \cref{theo:onenaturaltransformation_relative},
we obtain that $\giry^\omega$ carries a unique structure of lax symmetric monoidal functor:

\begin{cor}
	\label{cor:relativeproductlaxmonoidal}
  For any standard Borel space $B$ and $\omega \in \giry B$, the functor
  \[
    \giry^\omega\colon (\borelmeas/B,\times_B,B) \longrightarrow (\setc,\times,\{\ast\})
  \]
  carries a unique lax symmetric monoidal structure. Its structure maps are the relative product maps
  \begin{align*}
    \nabla^\omega \: \colon \: \giry^\omega X \times \giry^\omega Y &\longrightarrow \giry^\omega(X \times_B Y) \\
                                                             (p,q) &\longmapsto p \otimes_\omega q
  \end{align*}
  and the trivial map
  \begin{align*}
    \iota \: \colon \: \{\ast\} & \longrightarrow \giry^\omega B = \{\omega\} \\
		\ast & \longmapsto \omega.
  \end{align*}
\end{cor}

Note that the monoidal structures on both the source and target categories are the usual cartesian ones.

\begin{proof}
  The maps $\nabla^\omega$ are natural by \cref{lem:relativeproductnaturality},
	and it is clear from \cref{theo:onenaturaltransformation_relative} that they are the only natural maps of the correct type.
	Therefore at most one lax monoidal structure on $\giry^\omega$ can exist.

	To see that these structure maps indeed do the job, we first check associativity.
	With
  \[
    \alpha \: \colon \: (X \times_B Y) \times_B Z \longrightarrow X \times_B (Y \times_B Z)
  \]
	being the associator in $\borelmeas/B$,
	the two assignments 
  \begin{align*}
    (p,q,r) &\longmapsto \alpha_*((p \otimes_\omega q) \otimes_\omega r) \\
    (p,q,r) &\longmapsto p \otimes_\omega (q \otimes_\omega r)
  \end{align*}
  define maps
  \[
    \giry^\omega X \times \giry^\omega Y \times \giry^\omega Z \longrightarrow \giry^\omega(X \times_B (Y \times_B Z))
  \]
  which are natural in $X,Y,Z \in \borelmeas/B$.
	By the $n=3$ case of \cref{theo:onenaturaltransformation_relative}, these assignments agree. 
	This is the associativity condition for a lax monoidal functor.

	The unitality and symmetry conditions follow similarly from the $n=1$ and $n=2$ cases of \cref{theo:onenaturaltransformation_relative}, respectively.
\end{proof}

\subsection{Proof of \cref{theo:onenaturaltransformation}}

Although \cref{theo:onenaturaltransformation} is the special case of \cref{theo:onenaturaltransformation_relative} obtained by taking $B$ to be a singleton,
we prefer to give a direct proof in order to set the stage for the more general case of relative product measures.
Given a set $X$, we write $\Sym(X)$ for the symmetric group of $X$, which consists of bijections $\sigma: X \rightarrow X$.
We start the proof with the following basic observation.

\begin{lem}\label{lem:uniquepermutationinvariantmeasure}
  Let $X$ be a finite set and $G$ a group acting transitively on $X$. Then, there is exactly one $G$-invariant probability measure on $X$, namely the uniform measure $u_X$.
\end{lem}
\begin{proof}
  This is straightforward based on the fact that a bijection acts on measures by permuting the probabilities of individual outcomes.
\end{proof}

We call $p \in \dist X$ \emph{rational} if $p(x) \in \Q$ for all $x \in X$. A nice property of rational measures is that they are pushforwards of uniform measures on larger sets.
For the precise statement, recall our notation $\finiteset{N} \coloneqq \{1, \ldots, N\}$.

\begin{lem}\label{lem:rationalmeasure}
  Let $X$ be a finite set and $p \in \dist X$ rational. Then there is $N \in \N$ and $f:\finiteset{N} \rightarrow X$ such that $p = f_\ast u_{\finiteset{N}}$.
\end{lem}
\begin{proof}
	Let $X = \{x_1, \ldots, x_{k}\}$ and choose a common denominator $N$ for $p(x_1), \ldots, p(x_{k})$.
	This means that $p(x_j) = \frac{q_j}{N}$ for suitable $q_j \in \N$ with $q_1 + \ldots + q_{k} = N$.
  Define $f : \finiteset{N}  \rightarrow X$ by $f(a) \coloneqq x_{\ell(a)}$ for every $a \in \finiteset{N}$,
  where $\ell(a) \in \finiteset{k}$ is the unique index with $\sum_{j=1}^{\ell(a)-1} q_j < a \leq \sum_{j=1}^{\ell(a)} q_j$.
  Then we have for every $\ell \in \finiteset{k}$,
  \[
    f_\ast u_{\finiteset{N}}(x_\ell) = u_{\finiteset{N}}(f^{-1}(x_\ell)) 
		= u_{\finiteset{N}} \left(\left\{ \sum_{j=1}^{\ell-1} q_j + 1,\ldots, \sum_{j=1}^\ell q_j \right\}\right)
		= \frac{q_\ell}{N} = p(x_\ell).
  \]
  Thus, $p = f_\ast u_{\finiteset{N}}$.
\end{proof}

We can now use these two lemmas to already prove the claim for rational measures on finite sets.

\begin{prop}\label{prop:rationalcase}
  Let $\distfin$ be the distribution functor on $\finset$. Let $n \in \N$ and let
  \begin{equation}
    \label{eq:alpha}
    \alpha \: \colon \: \distfin X_1 \times \cdots \times \distfin X_n \Longrightarrow \distfin(X_1 \times \cdots \times X_n)
  \end{equation}
  be natural in $X_1, \dots, X_n$. Then,
  \[
    \alpha(p_1,\dots,p_n) = p_1 \otimes \cdots \otimes p_n
  \]
  for all rational $p_i \in \distfin X_i$.
\end{prop}

\begin{proof}
  Without loss of generality we assume that $n \geq 1$ and all $X_i$ are nonempty. For all permutations $\sigma_i \in \Sym(X_i)$, naturality of $\alpha$ gives
  \begin{align*}
    (\sigma_1 \times \cdots \times \sigma_n)_\ast\alpha(u_{X_1},\dots,u_{X_n}) &= \alpha(\sigma_{1,\ast} u_{X_1},\dots,\sigma_{n,\ast} u_{X_n}) \\[2pt]
                                                                          &= \alpha(u_{X_1},\dots,u_{X_n}),
  \end{align*}
  since uniform measures are invariant under permutations. Thus, $\alpha(u_{X_1},\dots,u_{X_n})$ is invariant under $\prod_{i=1}^n\Sym(X_i)$. This group acts transitively on $\prod_{i=1}^n X_i$. Hence, by \cref{lem:uniquepermutationinvariantmeasure},
  \[
    \alpha(u_{X_1},\dots,u_{X_n}) = u_{X_1 \times \cdots \times X_n}.
  \]

  Now let $p_i \in \distfin X_i$ be rational for every $i$. By \cref{lem:rationalmeasure}, there are  $N_1,\dots,N_n \in \N$ and maps
  \[
    f_i \: \colon \: \finiteset{N_i} \longrightarrow X_i
  \]
	such that $p_i = f_{i,\ast} u_{\finiteset{N_i}}$. Naturality of $\alpha$ and the uniform case give
  \begin{align*}
    \alpha(p_1,\dots,p_n) &= (f_1 \times \cdots \times f_n)_\ast\alpha(u_{\finiteset{N_1}},\dots,u_{\finiteset{N_n}}) \\[2pt]
                          &= (f_1 \times \cdots \times f_n)_\ast u_{\finiteset{N_1} \times \cdots \times \finiteset{N_n}}.
  \end{align*}
  By naturality of the product measure construction $\nabla$, hence,
  \begin{align*}
    \alpha(p_1,\dots,p_n) &= (f_1 \times \cdots \times f_n)_\ast (u_{\finiteset{N_1}} \otimes \cdots \otimes u_{\finiteset{N_n}}) \\[2pt]
													& = f_{1,\ast} u_{\finiteset{N_1}} \otimes \cdots \otimes f_{n,\ast} u_{\finiteset{N_n}} \\[2pt]
													& = p_1 \otimes \cdots \otimes p_n. \qedhere
  \end{align*}
\end{proof}

The general case without the restriction to rationality can now be obtained by a density argument. The elegance of this argument lies in translating the natural order on the interval $[0,1]$ into a purely probabilistic condition. This is very much in the spirit of \emph{effectus theory}, a categorical framework for probabilistic logic \cite{Jacobs2015}.

\begin{lem}\label{prop:order}
  Define $\theta^-,\theta^+: \three \rightarrow \two$ by
  \begin{align*}
	  \theta^-(1) = \theta^-(2) = 1, &&  \theta^+(1) = 1,  \\
	  \theta^-(3) = 2, && \theta^+(2) = \theta^+(3) = 2.
  \end{align*}
  Then for all $p, q \in \dist\two$, we have $p(1) \leq q(1)$ if and only if there is $\eta \in \dist\three$ such that
  \begin{equation}\label{eq:theta}
    \theta^+_\ast\eta = p \qquad \text{and} \qquad \theta^-_\ast\eta = q.
  \end{equation}
  In this case,
	there is only one such $\eta$, namely
  \begin{equation}
	  \label{eq:theta2}
	  \eta(1) = p(1), \qquad \eta(2) = q(1)-p(1), \qquad \eta(3) = q(2).
  \end{equation}
\end{lem}
\begin{proof}
  It is straightforward that the only solution of the linear system \cref{eq:theta} is given by \cref{eq:theta2} and that it satisfies the normalization condition $\eta(1) + \eta(2) + \eta(3) = 1$.
  It is nonnegative if and only if $p(1) \leq q(1)$.
\end{proof}

If we identify $\dist \two \cong [0,1]$ via $p \mapsto p(1)$, the relation $p(1) \leq q(1)$ becomes the natural order on $[0,1]$.
Hence  the natural order on $[0,1]$ induces an order on $\dist\two$ which is witnessed, in purely probabilistic terms, by $\eta$.

\begin{lem}\label{lem:bernoulliirrational}
  For all $p_1,\dots,p_n \in \dist\two$, every natural transformation $\alpha$ as in~\eqref{eq:alpha} satisfies
  \[
    \alpha(p_1,\dots,p_n) = p_1 \otimes \cdots \otimes p_n.
  \]
\end{lem}

\begin{proof}
  Define
  \begin{align*}
    \phi\colon \dist\two \times \cdots \times \dist\two &\longrightarrow [0,1] \\
                                       (p_1,\dots,p_n) &\longmapsto \alpha(p_1,\dots,p_n)(\{1\}^n).
  \end{align*}
  We first show that $\phi$ is monotonically nondecreasing in the product order. Let $p_i \leq p_i'$ for every $i$. As per \cref{prop:order}, these inequalities are witnessed by $\eta_i \in \dist \three$ with
  \[
    \theta^+_\ast\eta_i = p_i \qquad \text{and} \qquad \theta^-_\ast\eta_i = p_i'.
  \]
  Naturality of $\alpha$ gives
  \begin{align*}
    \phi(p_1',\dots,p_n')-\phi(p_1,\dots,p_n) &= \alpha(\theta^-_\ast\eta_1,\dots,\theta^-_\ast\eta_n)(\{1\}^n) - \alpha(\theta^+_\ast\eta_1,\dots,\theta^+_\ast\eta_n)(\{1\}^n)\\
                                              &= \alpha(\eta_1,\dots,\eta_n)(\{1,2\}^n) - \alpha(\eta_1,\dots,\eta_n)(\{1\}^n) \\
                                              &= \alpha(\eta_1,\dots,\eta_n)(\{1,2\}^n \setminus \{1\}^n) \geq 0.
  \end{align*}
  Hence, $\phi$ is monotonically nondecreasing.

  Now let $p_1,\dots,p_n \in \dist\two$ be arbitrary. By density of $\Q$ in $\R$, for every $i \leq n$, we can find rational $(p_{i,m}^-)_{m \in \N}$ and $(p_{i,m}^+)_{m \in \N}$ in $\dist\two$ with $p_{i,m}^-\nearrow p_i$ and $p_{i,m}^+ \searrow p_i$ as $m \to \infty$.
  By monotonicity and \cref{prop:rationalcase}, we obtain
  \[
    p_{1,m}^-(1) \cdots p_{n,m}^-(1) \leq \phi(p_1,\dots,p_n) \leq p_{1,m}^+(1) \cdots p_{n,m}^+(1).
  \]
  Taking $m \rightarrow \infty$ shows
  \[
    \phi(p_1,\dots,p_n) = p_1(1) \cdots p_n(1).
  \]

  Finally, let $(x_1,\dots,x_n) \in \two^n$. For every $i$, choose $\sigma_i \in \Sym(\two)$ such that $\sigma_i(x_i)=1$. By naturality of $\alpha$ and the equality just proved,
\begin{align*}
  \alpha(p_1,\dots,p_n)(x_1,\dots,x_n) &= \phi(\sigma_{1,\ast} p_1,\dots,\sigma_{n,\ast} p_n) \\
                                       &= (\sigma_{1,\ast} p_1)(1) \cdots (\sigma_{n,\ast} p_n)(1) \\
                                       &= p_1(x_1) \cdots p_n(x_n).
\end{align*}
Thus, $\alpha(p_1,\dots,p_n)$ and $\bigotimes_{i=1}^n p_i$ agree on every atom of $\two^n$, which proves the claim.
\end{proof}

The final steps are now relatively straightforward.

\begin{proof}[Proof of \cref{theo:onenaturaltransformation}]
  Let
  \[
    \alpha \: \colon \: \giry X_1 \times \cdots \times \giry X_n \longrightarrow \giry(X_1 \times \cdots \times X_n)
  \]
  be any family of maps natural in $X_1,\dots,X_n \in \meas$.
	By \cref{lem:bernoulliirrational}, $\alpha$ and $\nabla$ agree on $\prod_{i=1}^n \giry\two$.

  Let $X_1,\dots,X_n$ be arbitrary measurable spaces and let $p_i \in \giry X_i$ be given. For any $T_i \in \Sigma_{X_i}$, consider the pushforwards along the indicator functions,
  \[
    p_{i,T_i} \coloneq \chi_{T_i,\ast} p_i.
  \]
  Then we already know
  \[
    \alpha(p_{1,T_1},\dots,p_{n,T_n}) = p_{1,T_1} \otimes \cdots \otimes p_{n,T_n}
  \]
  by the finite case. Therefore, and by naturality of $\alpha$,
  \begin{align*}
    \alpha(p_1,\dots,p_n)(T_1 \times \cdots \times T_n) &= \bigl((\chi_{T_1} \times \cdots \times \chi_{T_n})_\ast\alpha(p_1,\dots,p_n)\bigr)(\{1\}^n) \\[2pt]
                                                        &= \alpha(p_{1,T_1},\dots,p_{n,T_n})(\{1\}^n) \\[2pt]
                                                        &= p_{1,T_1}(1) \cdots p_{n,T_n}(1) \\[2pt]
                                                        &= p_1(T_1) \cdots p_n(T_n)
  \end{align*}
  The claimed $\alpha = \nabla$ hence follows as before.
  
  By restricting the proof from $\giry$ to $\dist$ and $\distfin$,
	we also obtain the analogous statement for the latter two functors.
\end{proof}

\subsection{Proof of \cref{theo:onenaturaltransformation_relative}}

We follow the same steps as for \cref{theo:onenaturaltransformation}: first measures with uniform fibers, then rational probabilities,
then an order-theoretic approximation argument on $\two$, and finally indicator functions.
The only difference is that we have to work over the base space $B$ and with the base measure $\omega$ throughout.
So let
\[
  \alpha \: \colon \: \giry^\omega X_1 \times \cdots \times \giry^\omega X_n \longrightarrow \giry^\omega(X_1 \times_B \cdots \times_B X_n)
\]
be natural in $X_1,\dots,X_n \in \borelmeas/B$. Our goal is to prove that $\alpha = \nabla^\omega$.

For a finite set $X$ and a Markov kernel $p : B \to \dist X$, write
\[
  \overline{p} \coloneq \int_{b \in B} p(-|b) \otimes \delta_b \, \omega(\dx b) \quad\in\quad \giry^\omega(X \times B),
\]
where $X \times B$ is regarded as an object over $B$ via the projection. For example with $u_X$ the uniform measure on $X$, and with every measure regarded as a constant kernel, we get $\overline{u_X} = u_X \otimes \omega$. In general, we have
\[
  \overline{p}(\{x\} \times S) = \int_{b \in S} p(x|b)\,\omega(\dx b) \qquad \forall S \in \Sigma_B.
\]
This implies that the kernel $b \mapsto p(-|b) \otimes \delta_b$ is a regular conditional probability of $\overline{p}$ with respect to the projection to $B$.

The analogue of rational measures that we need consists of kernels whose probabilities have a common denominator over the whole base:

\begin{prop}\label{prop:relative_rationalcase}
  Let $X_1,\dots,X_n$ be nonempty finite sets and $p_i\colon  B \to \dist X_i$ Markov kernels.
	Suppose that there are $N_1, \dots, N_n \in \N$ such that
  \[
    N_ip_i(x|b) \in \N
		\qquad
		\forall x \in X_i, b \in B, i \in \finiteset{n}.
  \]
  Then
  \[
    \alpha(\overline{p_1},\dots,\overline{p_n}) = \overline{p_1} \otimes_\omega \cdots \otimes_\omega \overline{p_n}.
  \]
\end{prop}

From now on, we will freely use the canonical identification
\[
  (X_1 \times B) \times_B \cdots \times_B (X_n \times B) \cong (X_1 \times \cdots \times X_n) \times B,
\]
which frequently allows for more convenient notation in the proofs.

\begin{proof}
  We first consider the uniform measures $u_{X_1}, \dots, u_{X_n}$. Naturality under $\sigma_i \times \id_B$, for permutations $\sigma_i \in \Sym(X_i)$, implies that
  \[
    r \coloneq \alpha(\overline{u_{X_1}}, \dots, \overline{u_{X_n}})
  \]
  is invariant under $(\prod_{i=1}^n \sigma_i) \times \id_B$. For every $S \in \Sigma_B$, transitivity of the action on $\prod_{i=1}^n X_i$ therefore makes $r(\{(x_1,\dots,x_n)\} \times S)$ independent of $(x_1,\dots,x_n)$. Since $r$ has pushforward $\omega$ on $B$, the only possible value is
  \[
    r(\{(x_1,\dots,x_n)\} \times S) = \frac{\omega(S)}{|X_1| \cdots |X_n|}.
  \]
  Thus $r = u_{\prod_{i=1}^n X_i} \otimes \omega = \overline{u_{X_1}} \otimes_\omega \cdots \otimes_\omega \overline{u_{X_n}}$.

  Now let $p_1,\dots,p_n$ satisfy the stated denominator conditions. As in \cref{lem:rationalmeasure},
	write $X_i = \{x_{i,1},\ldots,x_{i,k_i}\}$ with $k_i = |X_i|$ and define maps over $B$ by
  \begin{align*}
    f_i \: \colon \: \finiteset{N_i} \times B & \longrightarrow X_i \times B \\
                                         (a,b) & \longmapsto (x_{i,\ell(a,b)},b),
  \end{align*}
  where $\ell(a,b)$ is the unique index with
  \[
    \sum_{j=1}^{\ell(a,b)-1} N_i p_i(x_{i,j}|b) < a \leq \sum_{j=1}^{\ell(a,b)} N_i p_i(x_{i,j}|b).
  \]
  The defining inequalities are measurable in $b$, and therefore each $f_i$ is measurable.
	For each $b$ and $\ell$, exactly $N_i p_i(x_{i,\ell}|b)$ elements map to $(x_{i,\ell},b)$, and hence
  \[
		f_{i,\ast} \overline{u_{\finiteset{N_i}}} = \overline{p_i}.
  \]
  By naturality, the uniform case, and \cref{lem:relativeproductnaturality}, we have
  \begin{align*}
    \alpha(\overline{p_1},\dots,\overline{p_n}) &= (f_1 \times_B \cdots \times_B f_n)_\ast \alpha(\overline{u_{\finiteset{N_1}}},\dots,\overline{u_{\finiteset{N_n}}}) \\[2pt]
                                                &= (f_1 \times_B \cdots \times_B f_n)_\ast (\overline{u_{\finiteset{N_1}}} \otimes_\omega \cdots \otimes_\omega \overline{u_{\finiteset{N_n}}}) \\[2pt]
                                                &= \overline{p_1} \otimes_\omega \cdots \otimes_\omega \overline{p_n},
  \end{align*}
	as was to be shown.
\end{proof}

We next extend from these rational kernels to arbitrary Bernoulli kernels, in parallel with \cref{lem:bernoulliirrational}. This means that for kernels $p,p'\colon B \to \dist\two$, we write
\[
	p \le p' \qquad \Longleftrightarrow \qquad
	p(1|b) \leq p'(1|b) \quad \forall b \in B.
\]
The witness in \cref{prop:order} can be chosen measurably: define $\eta : B \to \dist\three$ by
\[
	\eta(1|b) = p(1|b), \qquad
	\eta(2|b) = p'(1|b)-p(1|b), \qquad
	\eta(3|b) = p'(2|b).
\]
Writing $\theta_B^\pm \coloneq \theta^\pm \times \id_B$ then gives
\[
	\theta_{B,\ast}^+\overline{\eta} = \overline{p},
	\qquad
	\theta_{B,\ast}^-\overline{\eta} = \overline{p}'.
\]

\begin{lem}\label{lem:relative_bernoulliirrational}
For all Markov kernels $p_1,\dots,p_n\colon B \to \dist\two$, we have
  \[
    \alpha(\overline{p_1},\dots,\overline{p_n}) = \overline{p_1} \otimes_\omega \cdots \otimes_\omega \overline{p_n}.
  \]
\end{lem}

\begin{proof}
  For $S \in \Sigma_B$, define
  \[
    \phi_S(p_1,\dots,p_n) \coloneq \alpha(\overline{p_1},\dots,\overline{p_n})(\{1\}^n \times S).
  \]
  We first show that $\phi_S$ is monotonically nondecreasing in the product order. Let $p_i \leq p_i'$ for every $i$. As above, these inequalities are witnessed by Markov kernels $\eta_i\colon B \to \dist \three$ with
  \[
    \theta_{B,\ast}^+\overline{\eta_i} = \overline{p_i} \qquad \text{and} \qquad \theta_{B,\ast}^-\overline{\eta_i} = \overline{p_i}'.
  \]
  Naturality of $\alpha$ gives
  \begin{align*}
    \phi_S(p_1',\dots,p_n')-\phi_S(p_1,\dots,p_n) &= \alpha(\overline{\eta_1},\dots,\overline{\eta_n})(\{1,2\}^n \times S) - \alpha(\overline{\eta_1},\dots,\overline{\eta_n})(\{1\}^n \times S) \\[2pt]
                                                  &= \alpha(\overline{\eta_1},\dots,\overline{\eta_n})\bigl((\{1,2\}^n \setminus \{1\}^n) \times S\bigr) \geq 0.
  \end{align*}
  Hence $\phi_S$ is monotonically nondecreasing.

  Now let $p_1,\dots,p_n$ be arbitrary.
	For every $i$, define measurable dyadic approximations by
  \[
    p_{i,m}^-(1|b) \coloneq 2^{-m}\lfloor 2^m p_i(1|b)\rfloor, \qquad p_{i,m}^+(1|b) \coloneq 2^{-m}\lceil 2^m p_i(1|b)\rceil,
  \]
  and $p_{i,m}^\pm(2|b) \coloneq 1-p_{i,m}^\pm(1|b)$. These satisfy $p_{i,m}^- \nearrow p_i$ and $p_{i,m}^+ \searrow p_i$  and all have common denominator~$2^m$. By monotonicity and \cref{prop:relative_rationalcase}, we obtain
  \[
		\int_{b \in S} p_{1,m}^-(1|b) \cdots p_{n,m}^-(1|b)\,\omega(\dx b) \leq \phi_S(p_1,\dots,p_n) \leq \int_{b \in S} p_{1,m}^+(1|b) \cdots p_{n,m}^+(1|b)\,\omega(\dx b).
  \]
  Taking $m \to \infty$ and using dominated convergence gives
  \[
    \phi_S(p_1,\dots,p_n) = \int_{b \in S} p_1(1|b) \cdots p_n(1|b)\,\omega(\dx b).
  \]

  Finally, let $(x_1,\dots,x_n) \in \two^n$. For every $i$, choose $\sigma_i \in \Sym(\two)$ such that $\sigma_i(x_i)=1$. By naturality under $\sigma_i \times \id_B$ and the equality just proved,
  \begin{align*}
    \alpha(\overline{p_1},\dots,\overline{p_n})(\{(x_1,\dots,x_n)\} \times S) &= \phi_S(\sigma_{1,\ast} p_1,\dots,\sigma_{n,\ast} p_n) \\[2pt]
                                                                              &= \int_{b \in S} (\sigma_{1,\ast} p_1)(1|b) \cdots (\sigma_{n,\ast} p_n)(1|b)\,\omega(\dx b) \\[2pt]
                                                                              &= \int_{b \in S} p_1(x_1|b) \cdots p_n(x_n|b)\,\omega(\dx b).
  \end{align*}
  Since evaluation on sets of the form $\{(x_1,\dots,x_n)\} \times S$ distinguishes measures on $\two^n \times B$, the claim follows.
\end{proof}

The final reduction again uses indicator functions, now paired with the structure maps to $B$.

      \begin{proof}[Proof of \cref{theo:onenaturaltransformation_relative}]
  Given arbitrary $X_1,\dots,X_n \in \borelmeas/B$, we show that
  \[
    \alpha(p_1,\dots,p_n) = p_1 \otimes_\omega \cdots \otimes_\omega p_n
  \]
  for all $p_i \in \giry^\omega X_i$. Choose regular conditional probabilities $p_i(-|b)$ using \cref{prop:regularconditionalprobability}. For $T_i \in \Sigma_{X_i}$, define measurable maps $f_{T_i}\colon X_i \to \two \times B$ over $B$ by
  \[
    f_{T_i}(x_i) \coloneq (\chi_{T_i}(x_i),\rho_{X_i}(x_i)),
  \]
  where $\chi_{T_i}$ is the indicator function.
	Also consider the Markov kernels $p_{i,T_i}\colon B \to \dist\two$ characterized by
  \[
    p_{i,T_i}(\{1\}|b) \coloneq p_i(T_i|b).
  \]
  Then the disintegration identity~\eqref{eq:disintegration} gives
  \[
    f_{T_i,\ast} p_i = \overline{p_{i,T_i}}.
  \]
  Naturality and \cref{lem:relative_bernoulliirrational} therefore imply
  \begin{align*}
		\alpha(p_1,\dots,p_n)\bigl(T_1 & \times \cdots \times T_n \bigr) \\[2pt]
		&\quad= \bigl((f_{T_1} \times_B \cdots \times_B f_{T_n})_\ast\alpha(p_1,\dots,p_n)\bigr)(\{1\}^n \times B) \\[2pt]
		&\quad= \alpha(\overline{p_{1,T_1}},\dots,\overline{p_{n,T_n}})(\{1\}^n \times B) \\[2pt]
		&\quad= \int_{b \in B} p_1(T_1|b) \cdots p_n(T_n|b) \, \omega(\dx b) \\[2pt]
    &\quad= (p_1 \otimes_\omega \cdots \otimes_\omega p_n)\bigl(T_1 \times \cdots \times T_n\bigr).
  \end{align*}
  Since (the restrictions of) measurable rectangles form a generating $\pi$-system for $X_1 \times_B \cdots \times_B X_n$,
	we conclude the required $\alpha(p_1,\dots,p_n) = p_1 \otimes_\omega \cdots \otimes_\omega p_n$ from another application of the $\pi$-$\lambda$ theorem.

  Again, the restriction of the whole proof to $\dist^\omega$ and $\distfin^\omega$ is straightforward.
\end{proof}

\printbibliography[heading=bibintoc]
\end{document}

\typeout{get arXiv to do 4 passes: Label(s) may have changed. Rerun}
%%% Local Variables:
%%% mode: LaTeX
%%% TeX-master: t
%%% End':